\documentclass[12pt, a4paper]{amsart}
\usepackage{url,graphicx}
\usepackage{multirow, xcolor}
\usepackage{fullpage}
\usepackage{amsfonts,amscd,amssymb,amsmath,amsthm,mathrsfs,multirow,lscape,xfrac}
\usepackage{pbox,lipsum, stmaryrd,url, epstopdf,float, enumerate}
\usepackage{array}
\usepackage{xspace, comment}
\usepackage[all,cmtip]{xy}
\usepackage[english]{babel}
\usepackage{color}

\usepackage{ucs}
\usepackage[utf8]{inputenc}
\usepackage{dsfont}
\usepackage[T1]{fontenc} 
\usepackage{lmodern}
\usepackage[thinlines]{easytable}

\usepackage{phaistos}
\usepackage{microtype}

\usepackage{enumitem}
\usepackage{scalerel}
\usepackage{todonotes}
\usepackage{stackengine,wasysym}
\usepackage{todonotes}
\usepackage[T1]{fontenc}
\usepackage[colorlinks,linkcolor=blue,anchorcolor=blue,citecolor=blue,backref=page]{hyperref} 
\usepackage{cleveref}
\usepackage[english]{babel}

\newtheorem{theorem}{Theorem}[section]
\newtheorem{cor}[theorem]{Corollary}
\newtheorem{lemma}[theorem]{Lemma}

\def\tareesidedbox#1{%
\setbox0=\hbox{$\vphantom{\zeta^{(k)}}#1$}%
\dimen0=\wd0
\advance\dimen0 by3pt
\rlap{\hbox{\vrule height11pt width.4pt depth2pt
\kern-.4pt
\vrule height11.4pt width\dimen0 depth-11pt
\kern-.4pt
\vrule height11pt width.4pt depth2pt}}%
\hbox to\dimen0{\hss$\vphantom{\zeta^{(k)}}#1$\hss}}

\title[ Lower bound for  heights of H\'{e}non maps]{On lower bounds for  canonical heights of the map {$\phi(x,y)=(y,x+y^d+b)$}}
\author{Jorge Mello}

\begin{document}

\date{\today}

\begin{abstract}
    We give a lower bound for the canonical height associated to  Hénon maps $\phi(X,Y)=(Y,X+Y^D+B)$ of non-periodic points when $D>2$ in the spirit of conjectures of Lang and Silverman. This is followed by an application and extends previous work  for $D=2$ in  \cite{Ingram1}.
\end{abstract}

\maketitle

\tableofcontents

\section{Introduction} In algebraic dynamics of dimension $2$, the so-called Hénon maps, of normalized form
$$
\phi(x,y):=(ay,x+f(y))
$$ with $f(y)$ a polynomial of degree at least $2$, are automorphisms that play an important role in plane algebraic dynamics. The study of the arithmetic properties of such maps over number fields was initiated by Silverman \cite{Silverman}, who constructed canonical heights and proved that the set of corresponding periodic points  is finite. This was extended by several authors \cite{Ingram1,Kawaguchi2,Kawaguchi}. The canonical heights associated to polynomial automorphisms  of $\mathbb{A}^2$ were constructed by Kawaguchi ( see \cite{Kawaguchi}). For Hénon maps of degree $d$ over a number field or function field $K$ with set of places $M_K$, using the notation $\phi^N=\underbrace{\phi \circ \ldots  \circ \phi}_{n \text{ times}}$ and denoting the usual logarithmic Weil height in $\mathbb{A}^2(\overline{K})$ by $h$, those canonical heights are constructed as
$$
\hat{h}^+_\phi(P)=\lim_{N \rightarrow \infty}\dfrac{h(\phi^N(P))}{d^N}, \text{   } \text{   }\text{   }\text{  }\hat{h}^-_\phi(P)=\lim_{N \rightarrow \infty}\dfrac{h(\phi^{-N}(P))}{d^N},
$$and
$$
\hat{h}_\phi(P)=\hat{h}^+_\phi(P)+\hat{h}^-_\phi(P).
$$This height vanishes precisely at periodic points as expected from such so-called canonical height functions. This is also the case for classical canonical heights associated with morphisms or elliptic curves.

Naturally, one can look for non-trivial lower bounds on the smallest positive values of the canonical heights associated to maps of a given family, in the spirit of conjectures of Lang and Silverman (See \cite{Silverman2} and \cite{Looper}). In this context, for the particular family $\phi(x,y)=(y,x+y^2+b)$ over a number field or a function field, Ingram \cite{Ingram1} proved that if $P$ is not a periodic point, then 
$$
\hat{h}_\phi(P) \geq \epsilon \max \{h(b),1\}
$$ for a constant $\epsilon >0$ depending on the places dividing $b.$ In this note, we generalize this bound to similar Hénon maps of any degree $d\geq 2$,  motivated by the case $d=2$ and its proof. Namely, we prove the following
\begin{theorem}
Let $\phi(x,y)=(y,x+y^d+b)$ with $d\geq 2$ be a H\'{e}non map defined over a number field or function field $K$ of characteristic 0. Then there exists $B \in \mathbb{Z}^+$ such that if $P \in \mathbb{A}^2(K)$ is not periodic for $\phi$ of period at most $B$, then
$$
\hat{h}_\phi(P)\geq
\epsilon\max \{h(b),1\}
$$ where $\epsilon >0$ is a constant depending on $d$ and the places dividing $b.$
\end{theorem}
As a consequence of the theorem, we can can give more evidence on a conjecture by Ingram \cite[Conjecture 1.7]{Ingram1} on the finiteness of unlikely intersections of fibers of distinct orbits for the H\'{e}non maps over function fields in the families treated in this note. Denoting the orbit of $P$ under $\phi$ by
$$
\mathcal{O}_\phi(P)=\{\phi^N(P): N \in \mathbb{Z} \},
$$we obtain
\begin{cor}Let $F=K(C)$, for $C/K$ a curve and $K$ a number field, let $b \in F$ with pole divisor $\eta \in Div(C),$ 
let $\phi(x,y)=(y,x+y^d+b)$ with $d\geq 2$ and let $P,Q \in \mathbb{A}^2(F)$ have distinct orbits under $\phi.$ For any $s\geq 1,$ there exist only finitely many $t \in C(K)$, $s$-integral with respect to $\eta,$  such that $\mathcal{O}_{\phi_t}(P(t))=\mathcal{O}_{\phi_t}(Q(t)).$

\end{cor}

In Section 2 we recall facts about local canonical heights and prove the preliminary lemmas that are necessary for the proofs, which are given in Section 3.
\section{Preliminary properties and bounds for the size of iterates }
Given 
$$
\phi(x,y)=(y,x+y^d+b)
$$ defined over a number field $K$. We start recalling some basic information about the theory of local canonical heights for such H\'{e}non maps. Taking $\lVert (x,y)\rVert_v=\max \{|x|_v,|y|_v \},$ the local canonical heights for $\phi$ are defined by
$$
\hat{\lambda}^+_{v,\phi}(P)=\lim_{N \rightarrow \infty}\dfrac{\log^+\lVert\phi^N(P)\rVert}{d^N}, \text{   } \text{   }\text{   }\text{  }\hat{\lambda}^-_{v,\phi}(P)=\lim_{N \rightarrow \infty}\dfrac{\log^+\lVert\phi^{-N}(P)\rVert}{d^N},
$$and
$$
\hat{\lambda}_{v,\phi}(P)=\hat{\lambda}^+_{v,\phi}(P)+\hat{\lambda}^-_{v,\phi}(P),
$$ for which one naturally has 
$\hat{\lambda}^+_{v,\phi}(\phi(P))=d\hat{\lambda}^+_{v,\phi}(P)$ and $\hat{\lambda}^-_{v,\phi}(\phi^{-1}(P))=d\hat{\lambda}^-_{v,\phi}(P)$ ( see \cite[Lemma 2.1]{Ingram1}). 

According to \cite{Kawaguchi2}, the canonical heights for H\'{e}non maps decompose as sums of local canonical heights, namely,
$$
\hat{h}^+_{v,\phi}(P)=\sum_{v \in M_K} \dfrac{[K_v:\mathbb{Q}_v]}{[K:\mathbb{Q}]}\hat{\lambda}^+_{v,\phi}(P)
$$and
$$
\hat{h}^-_{v,\phi}(P)=\sum_{v \in M_K} \dfrac{[K_v:\mathbb{Q}_v]}{[K:\mathbb{Q}]}\hat{\lambda}^-_{v,\phi}(P).
$$

We set 
\[
(r)_v=
\begin{cases}
r & \text{ if } v \text{ is archimedean},\\ 
1& \text{ otherwise}
\end{cases}
\]
and define neighbourhoods
$$
\mathcal{B}_v^+(\phi)=\{(x,y) \in \mathbb{A}^2(K): |y|_v^d>(3)_v\max\{|x|_v,|b|_v,1 \}
$$and
$$
\mathcal{B}_v^-(\phi)=\{(x,y) \in \mathbb{A}^2(K): |x|_v^d>(3)_v\max\{|y|_v,|b|_v,1 \}.
$$

We describe some properties of local heights in the following lemma, which is a straight-forward
generalization of \cite[Lemma 4.1]{Ingram1}
\begin{lemma} 
The set $\mathcal{B}_v^+(\phi)$ is closed under the action of $\phi$, and 
$$
\hat{\lambda}^+_{v,\phi}(x,y)=\log|y|_v +\epsilon^+(b,P,v)
$$ for $P=(x,y) \in \mathcal{B}_v^+(\phi),$ where $\epsilon^+=0$ for $v \in M_K^0$, and $-\log 3 \leq \epsilon^+ \leq \log 5/3$ otherwise. Similarly, the set $\mathcal{B}_v^-(\phi)$ is closed under the action of $\phi^{-1}$, and 
$$
\hat{\lambda}^-_{v,\phi}(x,y)=\log|x|_v+\epsilon^-(b,P,v)
$$ for $(x,y) \in \mathcal{B}_v^-(\phi),$ where $\epsilon^-=0$ for $v \in M_K^0$, and $-\log 3 \leq \epsilon^- \leq \log 5/3$ otherwise. 
\end{lemma}
\begin{proof}Let $P=(x,y) \in \mathcal{B}_v^+(\phi).$ Suppose first that $v$ is non-archimedean. Then, $|y|_v^d>|x|_v$, and $|y|_v^d>|b|_v$, and so
$$
|x+y^d+b|_v=|y^d|_v>|y|_v.  
$$Thus, $\phi(P)\in \mathcal{B}_v^+(\phi), \lVert\phi^N(P)\rVert_v=|y|_v^{d^N}$ and $\hat{\lambda}_v^+(P)=\log|y|_v$ as desired.

Now suppose that $v$ is archimedean.  Then, $\frac{1}{3}|y|^d>|x|_v$, and $\frac{1}{3}|y|^d>|b|_v$, and so
$$
\left( 1+\frac{2}{3}\right)|y^d|_v>|x+y^d+b|_v> \left( 1-\frac{2}{3}\right)|y^d|_v>|y|_v.  
$$Thus, $\phi(P)\in \mathcal{B}_v^+(\phi)$ and
    $$
\frac{5}{3}|y^d|_v>\lVert \phi(P) \rVert_v>  \frac{1}{3}|y^d|_v.$$ By induction and taking limits, we obtain
$$
\log 5/3 > \hat{\lambda}_v^+(P)-\log|y|_v >-\log 3.
$$ The proofs for the facts about $\hat{\lambda}_v^-(P)$ and $\phi^{-1}(P)$ are analogous.
\end{proof}
The next result shows that points outside of the neighbourhoods $\mathcal{B}_v^+(\phi)$ and $\mathcal{B}_v^-(\phi)$ group $v$-adically.
\begin{lemma}
Let $P=(x,y) \notin \mathcal{B}_v^+(\phi)\bigcup \mathcal{B}_v^-(\phi)$
Then
$$
\Vert (x,y) \rVert_v \leq (3)_v\max \{1,|b|_v \}^{1/d}.
$$Moreover, if $\phi^{-1}(P),\phi(P) \notin \mathcal{B}_v^+(\phi)\bigcup \mathcal{B}_v^-(\phi),$ then there are roots $\gamma_1^d=-b$ and $\gamma_2^d=-b$ such that
$$
|x-\gamma_2|_v,|y-\gamma_1|_v\leq (6\cdot2^{d-1})_v|d|_v^{-1}|b|^{(2-d)/d}_v.
$$
\end{lemma}
\begin{proof} For the first part of the lemma, if $|b|_v\leq 1$,  then $|y|^d_v\leq \max\{1,|x|_v \}$ and $|x|^d_v\leq \max\{1,|y|_v \}$ imply $|x|_v,|y|_v\leq (3)_v$. 
Thus, we only need to consider $|b|_v>1.$

Suppose also that $|y|_v=((3)_v|b|^{1/d}_v)^c$ for some $c>1.$ Then
$$
(3)_v^{dc}|b|_v^c=|y|_v^d\leq (3)_v\max \{|b|_v, |x|_v \},
$$and so $1<(3)_v^{dc-1}|b|_v^c\leq |x|_v$, which implies
$$
(3)_v^{d^2c-d}|b|_v^{dc}\leq|x|^d_v\leq (3)_v\max \{|b|_v,|y|_v \}=(3)_v|y|_v=(3)_v^{1+c}|b|^{c/d}_v
$$ and 
$$
(3)^{(d^2-1)c-d-1}_v|b|_v^{(d^2-1)c/d}\leq 1,
$$which is a contradiction for $|b|_v>1.$

The proof of $|x|_v\leq (3)_v |b|_v^{1/d}$ follows analogously.

For the second part of the lemma in the non-archimedean case, assuming $\phi(P),\phi^{-1}(P) \notin \mathcal{B}_v^+(\phi)\bigcup \mathcal{B}_v^-(\phi)$ and $|b|_v>1$, we obtain
$$
|y^d+b|_v\leq (2)_v\max \{|x|_v,|x+y^d+b|_v \}\leq (6)_v |b|_v^{1/d}.
$$
Let $\gamma$ be a $d$-th root of $-b$ that minimizes $|y-\gamma|_v$ and denote by $\gamma_1,...,\gamma_{d-1}$ all the other $d$-th roots of $-b.$ Then
$$
|\gamma-\gamma_i|_v\leq |(\gamma-y)+(y-\gamma_i)|_v\leq (2)_v |y-\gamma_i|_v \text{ for }i=1,2,...,d-1,
$$ and so
$$
|y-\gamma|_v\leq \dfrac{(6)_v|b|_v^{1/d}}{\prod_i^{d-1} |y-\gamma_i|_v}\leq \dfrac{(6\cdot2^{d-1})_v|b|_v^{1/d}}{\prod_i^{d-1} |\gamma-\gamma_i|_v}=\dfrac{(6\cdot2^{d-1})_v|b|_v^{1/d}}{|d|_v|\gamma|_v^{d-1}}=(6\cdot2^{d-1})_v|d|_v^{-1}|b|^{(2-d)/d}_v
$$


The same argument applied to $|x^d+b|_v\leq(2)_v\max\{|y|_v, |y-x^d-b|_v \}$ yields the claim.

If $|b|_v\leq 1$, then
$$
|y-\gamma|_v\leq (2)_v\max\{ |y|_v,|\gamma|_v\}\leq(6)_v
$$for any $\gamma$ which is a $d$-th root of $-b$, implying the claim.
    
\end{proof}
For the next lemmas, which are the main tools for the main proof, consider $P_j=(x_j,y_j)=\phi^j(P),$ and $[M,N]:=\{M,M+1,\cdots,N-1,N \}$ for any $M, N$ in $\mathbb{Z}.$
\begin{lemma}
    Let $I \subset [-M,M]$ such that $\#I \geq 2,$ and suppose that $v$ is archimedean, or non-archimedean with $|b|_v>1$. Then there exists a subset $J \subset I$ with $\#J \geq\dfrac{1}{d^4+2}\#I-1$ such that for all $i\neq j \in J,$
    $$
    |x_i-x_j|_v+|y_i-y_j|_v+ \lambda_v(b)\leq 3\cdot d^M\hat{\lambda}_{v,\phi}(P)+\alpha_v
    $$where by convention the inequality holds if $x_i=x_j$ or $y_i=y_j$, and where

    \[
\alpha_v=
\begin{cases}
B(d) & \text{ if } v \text{ is archimedean},\\ \\
6 \log |d|^{-1}_v & \text{ otherwise}
\end{cases}
\]
    with $B(d):=2\log \left(\dfrac{6(1152\cdot 2^d)^d}{d^{d+1}}\right).$
\end{lemma}
\begin{proof}
    First, if there is a subset $J_0 \subset I$ with $\#J_0\geq \dfrac{1}{d^4+2}\#I$ such that $P_j \in \mathcal{B}_v^+(\phi)$ for all $j \in J_0,$ then for all $j \in J_0$ we have
    $$
    \log|y_j|_v\leq \hat{\lambda}_{v,\phi}(P_j)+(\log 3)_v\leq d^M\hat{\lambda}_{v,\phi}(P)+(\log 3)_v
    $$ and also
    $$
    \log|x_j|_v=\log|y_{j-1}|_v\leq d^{M-1}\hat{\lambda}_{v,\phi}(P)+(\log 3)_v
    $$ as long as $j \neq \min J_0$ and $P_{j-1} \in \mathcal{B}_v^+(\phi).$ Moreover, 
    $$
    \lambda_v(b)\leq d\min\{|y_i|_v,|y_j|_v\}-(\log 3)_v \leq d^M\hat{\lambda}_{v,\phi}(P)+(\log 3)_v.
    $$Thus, for all $i \neq j \in J:=J_0\setminus \{\min J_0\}$, we have
   \begin{equation}\begin{split}
\log|x_i-x_j|_v+\log|y_i-y_j|_v + \lambda_v(b)&\leq \log \max \{|x_i|_v,|x_j|_v \}+ \log \max \{|y_i|_v,|y_j|_v \} \\
&\text{ } \text{ } \text{ }+\lambda_v(b)+ (\log 4)_v \\ & 
 \leq d^M\hat{\lambda}_{v,\phi}(P)+d^{M-1}\hat{\lambda}_{v,\phi}(P) \\
& \text{ } \text{ } \text{ } + d^M\hat{\lambda}_{v,\phi}(P)+ (\log 108)_v\\
& \leq 3d^M\hat{\lambda}_{v,\phi}(P)+ (\log 108)_v,
\end{split}
\end{equation} which proves the result in this case with such choice of $J.$ 

Analogously, one proves the result in case there exists $J_0 \subset I$ with $\#J_0\geq \dfrac{1}{d^4+2} \#I$ such that $P_j \in \mathcal{B}_v^-(\phi)$ for all $j \in J_0.$ 

Assume then that no such set exists. In this case, there is a subset $J_0 \subset I$ with $\#J_0\geq \dfrac{d^4}{d^4+2}\#I$ such that $P_j \notin \mathcal{B}_v^+(\phi)\cup \mathcal{B}^-_v(\phi)$ for all $j \in J_0.$ By Lemma 2.2, for each $j \in [ \min(J_0)+2,\max J_0-2]$, there exist
$$
\gamma_{1j}^d=\gamma_{2j}^d=\gamma_{3j}^d=\gamma_{4j}^d=-b
$$ such that
$$
|x_j-\gamma_{1j}|_v,|x_{j-1}-\gamma_{2j}|_v,|y_j-\gamma_{3j}|_v,|y_{j+1}-\gamma_{4j}|_v\leq (6\cdot2^{d-1})_v|d|_v^{-1}|b|^{(2-d)/d}_v
$$By the pigeonhole principle, there is a subset $J\subset J_0$ with
$$
\#J \geq \frac{1}{d^4}(\#J_0-4)\geq \frac{1}{d^4+2}\#I -1
$$such that $\gamma_{nj}$ is the same for each $n$ for all $j \in J.$ This implies
$$
|x_i-x_j|_v,|y_i-y_j|_v,|x_{i-1}-x_{j-1}|_v,|y_{i+1}-y_{j+1}|_v\leq (6\cdot2^{d})_v|d|_v^{-1}|b|^{(2-d)/d}_v
$$for all $i,j \in J.$ Then, for $i,j \in J$
\begin{equation}
    \begin{split}
    |y_i^d-y_j^d|_v&\leq (2)_v\max\{|y_i^d-y_j^d+x_i-x_j|_v,|x_i-x_j|_v \}\\
    &\leq (2)_v \max\{|y_{i+1}-y_{j+1}|_v,|x_i-x_j|_v \}\\
    &\leq (6\cdot2^{d+1})_v|d|_v^{-1}|b|^{(2-d)/d}_v.
\end{split}
\end{equation}

Also, if $v$ is non-archimedean and $|d|_v=1$, then
$$
|y_i^{d-1}+y_i^{d-2}y_j+...+y_iy_j^{d-2}+y_j^{d-1}|_v = |b|^{(d-1)/d}_v.$$ 
In fact, if $\epsilon:=y_j-y_i$, then $|\epsilon|_v\leq |b|_v^{(2-d)/d}< 1$ and since $y_i=y_j-\epsilon$ we have
$$
y_j^d-y_i^d=y_j^d-(y_j-\epsilon)^d=dy_j^{d-1}\epsilon-\binom{d}{2}y_j^{d-2}\epsilon^2+...\pm \epsilon^d
$$ whose absolute value comes from $dy_j^{d-1}\epsilon$, which is $|b^{(d-1)/d}\epsilon|_v$. This implies $$|(y_i^d-y_j^d)/\epsilon|_v=|b|^{(d-1)/d}_v$$ as claimed.

The claim yields $|y_i-y_j|_v\leq |b|^{(2-d)/d-(d-1)/d}_v=|b|_v^{\frac{3-2d}{d}}$. We analogously obtain $$|x_i-x_j|_v\leq |b|_v^{\frac{3-2d}{d}}.$$ This gives
$$
|x_i-x_j|_v+|y_i-y_j|_v+\lambda_v(b)\leq|x_i-x_j|_v+|y_i-y_j|_v+\dfrac{2(2d-3)}{d}\lambda_v(b)\leq 0 \text{ as desired.}
$$If $|d|_v<1,$ then $dy_j^{d-1}\epsilon$ will have the largest absolute value of the terms in the sum above for $y_i^d-y_j^d$ when $|dy_j|_v>|\epsilon|_v$, e.g., when $|b|^{1/d}_v>|d|^{-2}_v|b|^{(2-d)/d}_v$, which implies $|b|_v>|d|_v^{-2d/(d-1)}$. In this case, we have again $|(y_i^d-y_j^d)/\epsilon|_v=|d|_v|b|^{(d-1)/d}_v$ and
$$
|y_i-y_j|_v\leq  |d|^{-2}_v|b|^{(2-d)/d-(d-1)/d}_v=|d|_v^{-2}|b|_v^{\frac{3-2d}{d}}.
$$On the other hand, if $|b|^{1/d}_v\leq|d|^{-2}_v|b|^{(2-d)/d}_v,$ then the inequality
$|y_i-y_j|_v\leq|d|^{-1}_v|b|^{(2-d)/d}_v$ implies
$$
|y_i-y_j|_v\leq |d|^{-3}_v|b|^{2(2-d)/d}_v|b|_v^{-1/d}\leq |d|^{-3}_v|b|^{\frac{3-2d}{d}}_v.
$$Obtaining similar estimates for $|x_i-x_j|_v$, we have this time
$$
|x_i-x_j|_v+|y_i-y_j|_v+\lambda_v(b)\leq6\log |d|_v^{-1}.
$$Finally, suppose that $v$ is archimedean.  If we first consider $|b|_v$ sufficiently large, let's say, $|b|_v> C(d):= \left(\dfrac{144\cdot 2^d}{d} \right)^{d^2}.$ Then, since $C(d) > \left(\dfrac{24\cdot 2^{d-1}}{d} \right)^{d}$, we have
$$
|b|^{1/d}_v=|\gamma_{3j}|_v\leq |y_j|_v+|y_j-\gamma_{3j}|_v\leq |y_j|_v+\dfrac{6\cdot 2^{d-1}}{d}<|y_j|_v+\dfrac{|b|_v^{1/d}}{4},
$$and so
$|y_j|_v\geq  \dfrac{3|b|_v^{1/d}}{4}.$ Recalling that we can write
$$
S_d:=\dfrac{y_j^d-y_i^d}{y_j-y_i}=dy_j^{d-1}-\binom{d}{2}y_j^{d-2}\epsilon+...\pm \epsilon^{d-1}
$$with $\epsilon=y_j-y_i,$ the triangle inequality, Lemma 2.2, $|b|_v>1$ and the bounds above imply
\begin{equation}\begin{split}
|S_d|_v&\geq d|y_j^{d-1}|_v-\displaystyle\sum_{k=2}^{d} \binom{d}{k}|y_j|_v^{d-k}|\epsilon|^{k-1}_v\\
& \geq d|y_j|^{d-1}_v-\left(\displaystyle\sum_{k=2}^{d} \binom{d}{k}\right)(3|b|_v^{1/d})^{d-2}\left(\dfrac{6\cdot 2^d}{d} \right)^{d-1}|b|_v^{(2-d)(d-1)/d}\\
& \geq d\left( \dfrac{3|b|_v^{1/d}}{4}\right)^{d-1}-2^d(3|b|_v^{1/d})^{d-2}\left(\dfrac{6\cdot 2^d}{d} \right)^{d-1}|b|_v^{(2-d)(d-1)/d}\\
& \geq 3\dfrac{|b|_v^{(d-1)/d}}{4^{d-1}}-|b|_v^{(d-2)/d}\left(\dfrac{36\cdot 2^d}{d} \right)^{d-1}\\
&\geq \dfrac{|b|_v^{(d-1)/d}}{4^{d-1}}
\end{split}
\end{equation} for $|b|_v>C(d).$ Thus,
$$
|y_i-y_j|_v=\dfrac{|y_j^d-y_i^d|_v}{|S_d|_v}\leq \dfrac{6\cdot 2^{d+1}|b|_v^{(2-d)/d}}{d|S_d|_v}\leq \dfrac{6\cdot8^d|b|_v^{(2-d)/d-(d-1)/d}}{d}\leq \dfrac{6\cdot8^d|b|_v^{(3-2d)/d}}{d}.
$$Obtaining the same estimates for $|x_i-x_j|_v,$ we have
$$
\log |x_i-x_j|_v+\log |y_i-y_j|_v + \log|b|_v\leq 2\log \left(\dfrac{6\cdot8^d}{d} \right)
$$for all $i,j \in J$. 

If on the other hand $|b|_v\leq C(d)$, then $|b|_v^{-1/d}\geq C(d)^{-1/d}$ and by Lemma 2.2 we have
$$
|y_i-y_j|_v\leq \dfrac{6\cdot 2^d}{d}|b|_v^{(2-d)/d}\leq \dfrac{6\cdot 2^d}{d}C(d)^{1/d}|b|_v^{(2-d)/d-1/d}= \dfrac{6\cdot 2^d}{d}C(d)^{1/d}|b|_v^{(1-d)/d}.
$$Obtaining similar estimates for $|x_i-x_j|_v$, we conclude that $$
\log |x_i-x_j|_v+\log |y_i-y_j|_v + \log|b|_v\leq 2\log \left(\dfrac{6\cdot8^d}{d}C(d)^{1/d}\right)
$$for all $i,j \in J$ as desired since $\hat{\lambda}_{v,\phi}\geq 0$.
\end{proof}
\begin{lemma}
Let $I \subset [-M,M]$ such that $\#I \geq 2$ and $y_i=y_j=y$   for all $i,j \in I.$ Then there exists a subset $J \subset I$ with $\#J \geq \dfrac{1}{d^2+1}\#I-1$ such that for all $i,j \in J,$
$$
\log|x_i-x_j|_v+\dfrac{1}{2}\lambda_v(b) \leq d^{M+1}\hat{\lambda}_{v,\phi}^-(P)+\beta_v,
$$where by convention the inequality holds if $x_i=x_j$, where
\[
\beta_v=
\begin{cases}
\log ((6\cdot 2^{2d-1})d^{-2}) & \text{ if } v \text{ is archimedean},\\ \\
6 \log |d|^{-1}_v & \text{ otherwise}.
\end{cases}
\]
\end{lemma}
\begin{proof}
    First suppose that There is a subset $J \subset I$ such that $\#J \geq \dfrac{1}{d^2+1}\#I$ and $P_j \in \mathcal{B}_v^-(\phi)$ for all $j \in J.$ Then for all $i,j \in J$, we have
\begin{equation}
\begin{split}
   \log |x_i-x_j|_v + \dfrac{1}{d}\lambda_v(b) &\leq \log \max \{|x_i|_v,|x_j|_v \}+\log (2)_v\\
   &\leq 2 \log \max \{|x_i|_v,|x_j|_v \}+\left( \log 2- \dfrac{1}{d}\log 3 \right)_v\\
   &\leq  2 \log \max \{d^i\hat{\lambda}^-_{v,\phi}(P),d^j\hat{\lambda}^-_{v,\phi}(P) \}+ \left( \log 2- \dfrac{1}{d}\log 3 +2 \log 3\right)_v\\
   & \leq d^{M+1} \hat{\lambda}^-_{v,\phi}(P)+\left( \log \left( 2 \cdot 3^{2-1/d}\right)\right)_v.
\end{split}  
\end{equation}  In this case we are done, so we assume now that such $J$ does not exist.

Now we will prove that $P_i \notin \mathcal{B}_v^+(\phi)$ for all $i \in I$ except for $i \in \max(I)$ or if $v$ is archimedean and $|b|_v\leq 15^{d-1}\cdot 5.$ In fact, if $v$ is archimedean and $(x_{i_1},y),(x_{i_2},y)\in \mathcal{B}_v^+(\phi)$ and $i_2>i_1$ in $I$, then
$$
d^{i_2}\hat{\lambda}^+_{v,\phi}(P)=\hat{\lambda}^+_{v,\phi}(P_{i_2})=\log|y|_v=\hat{\lambda}^+_{v,\phi}(P_{i_1})=d^{i_1}\hat{\lambda}^+_{v,\phi}(P),
$$which implies $i_2=i_1$ and the claim in this case. On the other hand, if $v$ is archimedean and $i_2\geq i_1+1$ we have
\begin{equation}
\begin{split}
\log |y|_v &\leq \hat{\lambda}^+_{v,\phi}(P_{i_1}) + \log 3 \\
&= d^{i_1-i_2}\hat{\lambda}^+_{v,\phi}(P) + \log 3\\
& \leq d^{i_1-i_2}(\log |y|_v + \log 5/3) + \log 3\\
& \leq \dfrac{1}{2}\log|y|_v+\dfrac{1}{2}\log 15.
    \end{split}
\end{equation}Thus,
$$
\dfrac{1}{d}\log^+ |b|_v +\dfrac{1}{d}\log 3\leq \log|y|_v \leq \log 15
$$ and so $|b|_v\leq 15^d/3$ as claimed. In this case, we can choose a set $J \subset I$ with $\#J \geq \dfrac{d^2}{d^2+1}\#I$ such that $P_j \notin \mathcal{B}^-_v(\phi)$ for all $j \in J.$ For these $j$, we have
$$
d\log|x_j|_v\leq \log \max \{1,|b|_v,|y|_v \}+\log 3 \leq \log 15^d,
$$ and so
$$
\log |x_i-x_j|_v +\log^+|b|_v \leq \max \{|x_i|_v,|x_j|_v \}+ \log^+|b|_v +\log 2 \leq \log( 10 \cdot 15^d)
$$ for all $i,j \in J.$
This proves the claim in this case. 

We can now suppose $P_i \notin \mathcal{B}^+_v(\phi)$ except for $i \in \max(I).$

In this case, there exists a subset $J_0 \subset I$ with $\#J_0 \geq \dfrac{d^2}{d^2+1}\#I -1$ such that
$$
P_j \notin \mathcal{B}_v^+(\phi)\cup \mathcal{B}_v^-(\phi)
$$ for all $j \in J_0.$ By Lemma 2.2, we have that for each such $j,$
$$
|x_{j-1}-\gamma_j|_v\leq (6\cdot2^{d-1})_v|d|_v^{-1}|b|^{(2-d)/d}_v
$$ for roots $\gamma_j^d=-b.$ By the pigeonhole principle, we can choose $J_1 \subset J_0$ with
$$
\#J_1 \geq \dfrac{1}{d}(\#J_0-2)\geq \dfrac{d}{d^2+1}\#I -\dfrac{1}{d}-1=\dfrac{d}{d^2+1}\#I -\dfrac{d+1}{d}
$$ such that $|x_{j-1}-\gamma|_v\leq (6\cdot2^{d-1})_v|d|_v^{-1}|b|^{(2-d)/d}_v$ for all $j \in J_1$ for one particular $\gamma^d=-b$ that does not depend on $j.$ Then, for $j \in J_1,$ we have
$$
|x_j^d-(y-b-\gamma)|_v=|x_{j-1}-\gamma|_v\leq (6\cdot2^{d-1})_v|d|_v^{-1}|b|^{(2-d)/d}_v.
$$ Suppose first that $v$ is non-archimedean. In this case, if $\delta^d=y-b-\gamma,$ then$|\delta|_v=|b|_v^{1/d}$. If we suppose that $\delta$ is the $d$-th root of $y-b-\gamma$ that minimizes $|\delta-x_j|_v$ and denote by $\delta_1,\delta_2,...,\delta_{d-1}$ the other roots of $y-b-\gamma$, then
$$
|\delta-\delta_i|_v\leq |(\delta-x_j)+(x_j-\delta_i)|_v\leq (2)_v |x_j-\delta_i|_v \text{ for }i=1,2,...,d-1,
$$ and so
\begin{equation}
\begin{split}
|x_j-\delta|_v\leq \dfrac{(6\cdot2^{d-1})_v|d|_v^{-1}|b|^{(2-d)/d}_v}{\prod_i^{d-1} |x_j-\delta_i|_v}&\leq \dfrac{(6\cdot2^{2d-2})_v|d|_v^{-1}|b|^{(2-d)/d}_v}{\prod_i^{d-1} |\delta-\delta_i|_v}\\&=\dfrac{|d|_v^{-1}|b|^{(2-d)/d}_v}{|d|_v|\delta|_v^{d-1}}\\&=|d|_v^{-2}|b|^{(3-2d)/d}_v.
\end{split} 
\end{equation} We can now choose a subset $J \subset J_1$ with $\#J\geq \#J_1/d\geq \dfrac{1}{d^2+1}\#I -\dfrac{d+1}{d^2}\geq \dfrac{1}{d^2+1}\#I -1$ such that 
$$
|x_i-x_j|_v\leq |d|_v^{-2}|b|^{(3-2d)/d}_v
$$ for all $i,j \in J.$ For all such $i,j$, we have
$$
\log|x_i-x_j|_v+\dfrac{1}{2}\lambda_v(b)\leq
\log|x_i-x_j|_v+\dfrac{2d-3}{d}\lambda_v(b) \leq 2\log |d|_v^{-1}\leq d^{M+1}\hat{\lambda}_{v,\phi}^-(P)+2\log |d|_v^{-1}.
$$Finally, suppose that $v$ is archimedean, $|b|_v \geq 15^d/3$ and $\delta$ is as above. Since $|\gamma|_v=|b|_v^{1/d}$ and $|y|_v\leq 3 |b|_v^{1/d}$ by Lemma 2.2, we have 
$$
|y-b-\gamma|_v\geq |b|_v-4|b|_v^{1/d}\geq |b|_v/2
$$ as $|b|_v\geq 75.$ Thus, $|\delta|_v\geq \dfrac{|b|_v^{1/d}}{2^{1/d}}$. Together with the computations from equation (6), this  implies
$$
|x_j-\delta|_v\leq (6\cdot 2^{2d-1})d^{-2}|b|^{(3-2d)/d}_v,
$$ and again a set $J \subset I$ with $\#J \geq \dfrac{1}{d^2+1}\#I -1$ such that for all $i,j \in J$ we have
$$
\log|x_i-x_j|_v+\dfrac{1}{2}\lambda_v(b) \leq d^{M+1}\hat{\lambda}_{v,\phi}^-(P)+\log ((6\cdot 2^{2d-1})d^{-2}).
$$
\end{proof}Analogously, a modification of the previous lemma yields
\begin{lemma}
Let $I \subset [-M,M]$ such that $\#I \geq 2$ and $x_i=x_j$   for all $i,j \in I.$ Then there exists a subset $J \subset I$ with $\#J \geq \dfrac{1}{d^2+1}\#I-1$ such that for all $i,j \in J,$
$$
\log|y_i-y_j|_v+\dfrac{1}{2}\lambda_v(b) \leq d^{M+1}\hat{\lambda}_{v,\phi}^+(P)+\beta_v,
$$where by convention the inequality holds if  $y_i=y_j$.
\end{lemma} Moreover, for good reduction primes, we have
\begin{lemma}
    Suppose that $|b|_v\leq 1$ and $v$ is non-archimedean. Then for any $i,j \in [-M,M],$ we have
    $$
    \log|x_i-x_j|_v\leq d^{M+1}\hat{\lambda}_{v,\phi}(P)
    $$and
    $$
    \log |y_i-y_j|_v\leq d^{M+1}\hat{\lambda}_{v,\phi}(P).
    $$
\end{lemma}
\begin{proof}
If $P_i=(x_i,y_i) \in \mathcal{B}^-_v(\phi),$ then
$$
\log|x_i|_v=\hat{\lambda}_{v,\phi}^-(P_i)\leq d^{M+1}\hat{\lambda}_{v,\phi}^-(P).
$$If $P_i \notin \mathcal{B}^-_v(\phi)$, then $|x_i|_v^d \leq \max \{1,|y_i|_v \}$. If $|y_i|_v\leq 1$ then
$$
\log|x_i|_v\leq 0 \leq d^{M+1}\hat{\lambda}_{v,\phi}^-(P).
$$If, on the other hand, $|y_i|_v>1$, then $|y_i|_v^d\geq \{1,|x_i|_v\}$ and so $P \in \mathcal{B}_v^{+}(\phi.)$ In this case,
$$
\log|x_i|_v\leq d\log |y_i|_v=d \hat{\lambda}_{v,\phi}^+(P_i)\leq d^{M+1}\hat{\lambda}_{v,\phi}^+(P).
$$Therefore,
$$
|x_i-x_j|_v\leq \max \{|x_i|_v,|x_j|_v \}\leq d^{M+1}(\hat{\lambda}_{v,\phi}^+(P)+\hat{\lambda}_{v,\phi}^-(P)).
$$The other inequality is deduced  similarly.
\end{proof}
\section{Proofs of theorem 1.1 and corollary 1.2}

We follow the ideas in \cite[Section 4]{Ingram1}. Let $B_{0,0}=d$ and define $B_{0,n+1}=(d^2+1)B_{0,n}+ d^2+1$ and $B_{m+1,n}=(d^4+2)B_{m,n}+d^4+2.$ Let $S$ be the set of places of $K$ dividing $b$ and $\#S=s$  so that $|b|_v\leq 1$ for the places $v$ that are not in $S$, i.e., there are at most $s$ places of bad reduction. We consider a positive integer $M$ such that $2M \geq B_{s,s}$. By Lemma 2.3 applied to places of bad reduction and Lemma 2.6 applied to places of good reduction, we can choose a subset $I \subset [-M,M]$ with $\#I \geq B_{0,s}$ such that for all $i,j \in I$ and all $v \in M_K,$ we have
$$
    |x_i-x_j|_v+|y_i-y_j|_v+ \lambda_v(b)\leq  d^{M+2}\hat{\lambda}_{v,\phi}(P)+\alpha_v.
    $$
    If there exist two values $i,j \in I$ such that $x_i\neq x_j, y_i\neq y_j,$ then summing the inequality above over all places of $K$ with the appropriate weights results in
    $$
    h(b)\leq d^{M+2}\hat{h}_\phi(P) +C
    $$where $C$ depends only on $d.$ 
    
    Otherwise, we either have $x_i=x_j$ for all $i, j \in I$ or $y_i=y_j$ for all $i,j \in I.$ If we suppose that the former is true, then $y_i\neq y_j$, and Lemma 2.5 together with Lemma 2.6 guarantee the existence of a subset $J \subset I$ with $\#J \geq 2$ such that
    $$
\log|y_i-y_j|_v+\dfrac{1}{2}\lambda_v(b) \leq d^{M+1}\hat{\lambda}_{v,\phi}^+(P)+\beta_v,
$$ for all $i,j \in J, v \in M_K.$ Then summing the inequality above over all places of $K$ with the appropriate weights results in
    $$
    h(b)\leq d^{M+2}\hat{h}^+_\phi(P) +C\leq d^{M+2}\hat{h}_\phi(P) +C
    $$where $C$ depends only on $d.$ The case with $ y_i=y_j$ for all $i,j \in I$ is handled similarly. Now, the fact that $\hat{h}_\phi$ is discrete \cite{Kawaguchi2} yields
    $$
    h(b)\leq d^{M+2}\cdot\min \{\hat{h}_\phi(P) |P \in K^2, \hat{h}_\phi(P) \neq 0\} +C,
    $$ which proves Theorem 1.1. If the latter was true, we use Lemma 2.4 instead of Lemma 2.5 to reach the same conclusion

    Corollary 1.2 can be proven exactly the same as in the proof of \cite[Theorem 1.8]{Ingram1}, but replacing the usage of \cite[Theorem 1.4]{Ingram1} therein by Theorem 1.1 of this note.
   
\section*{Acknowledgements}
This work was supported by the Oakland University and the FCT grant.
\nocite{}
\bibliographystyle{abbrv}
\bibliography{Castle}

@article {Silverman,
    AUTHOR = {Silverman, Joseph},
     TITLE = {Geometric and arithmetic properties of the \text{H}\'{e}non map},
   JOURNAL = {Math. Z.},
  FJOURNAL = {Math. Z.},
    VOLUME = {215},
      YEAR = {1994},
    NUMBER = {2},
     PAGES = {237--250},
      ISSN = {0065-1036,1730-6264},
   MRCLASS = {11R18},
  MRNUMBER = {3119783},
MRREVIEWER = {Jos\'e\ Othon Dantas Lopes},
       DOI = {10.4064/aa160-4-2},
       URL = {https://doi.org/10.4064/aa160-4-2},
}

@article {Silverman2,
    AUTHOR = {Silverman, Joseph},
     TITLE = {Lower bound for the canonical height on elliptic curves},
   JOURNAL = {Duke Math.
J.},
  FJOURNAL = {Duke Math.
J.},
    VOLUME = {48},
      YEAR = {1981},
    NUMBER = {},
     PAGES = {633--648},
      ISSN = {0065-1036,1730-6264},
   MRCLASS = {11R18},
  MRNUMBER = {3119783},
MRREVIEWER = {Jos\'e\ Othon Dantas Lopes},
       DOI = {10.4064/aa160-4-2},
       URL = {https://doi.org/10.4064/aa160-4-2},
}

@article {Ingram1,
    AUTHOR = {Ingram, Patrick},
     TITLE = {Canonical maps for \text{H}\'{e}non maps},
   JOURNAL = {Proceedings of the London Math. Soc.},
  FJOURNAL = {Proceedings of the London Math. Soc.},
    VOLUME = {108},
      YEAR = {2014},
    NUMBER = {3},
     PAGES = {780--808},
      ISSN = {0065-1036,1730-6264},
   MRCLASS = {11R18},
  MRNUMBER = {3119783},
MRREVIEWER = {Jos\'e\ Othon Dantas Lopes},
       DOI = {10.4064/aa160-4-2},
       URL = {https://doi.org/10.4064/aa160-4-2},
}

@article {Kawaguchi2,
    AUTHOR = {Kawaguchi, Shu},
     TITLE = {Canonical height functions for affine plane automorphisms},
   JOURNAL = {Math. Ann.},
  FJOURNAL = {Math. Ann.},
    VOLUME = {335},
      YEAR = {2006},
    NUMBER = {2},
     PAGES = {285-310},
      ISSN = {0065-1036,1730-6264},
   MRCLASS = {11R18},
  MRNUMBER = {3119783},
MRREVIEWER = {Jos\'e\ Othon Dantas Lopes},
       DOI = {10.4064/aa160-4-2},
       URL = {https://doi.org/10.4064/aa160-4-2},
}

@article {Kawaguchi,
    AUTHOR = {Kawaguchi, Shu},
     TITLE = {Local and global canonical height functions for affine space regular automorphisms},
   JOURNAL = {Algebra \& Number Theory},
  FJOURNAL = {Algebra \& Number Theory},
    VOLUME = {7},
      YEAR = {2013},
    NUMBER = {5},
     PAGES = {},
      ISSN = {0065-1036,1730-6264},
   MRCLASS = {11R18},
  MRNUMBER = {3119783},
MRREVIEWER = {Jos\'e\ Othon Dantas Lopes},
       DOI = {10.4064/aa160-4-2},
       URL = {https://doi.org/10.4064/aa160-4-2},
}

@article {Looper,
    AUTHOR = {Looper, Nicole},
     TITLE = {A lower bound on the canonical height for polynomials},
   JOURNAL = {Math. Ann.},
  FJOURNAL = {Math. Ann.},
    VOLUME = {373},
      YEAR = {2019},
    NUMBER = {},
     PAGES = {1057-1074},
      ISSN = {0065-1036,1730-6264},
   MRCLASS = {11R18},
  MRNUMBER = {3119783},
MRREVIEWER = {Jos\'e\ Othon Dantas Lopes},
       DOI = {10.4064/aa160-4-2},
       URL = {https://doi.org/10.4064/aa160-4-2},
}
\end{document}